\documentclass[reqno]{amsart}  
\usepackage{graphicx}
\usepackage[section]{algorithm}
\usepackage{algorithmic} 
\usepackage{amssymb, amsmath}
\usepackage{cite}

\usepackage{helvet}
\usepackage{courier}
\usepackage{type1cm}

\numberwithin{equation}{section}

\newcommand{\BlackBoxes}{\global\overfullrule5pt}

\BlackBoxes

\newcommand{\weakD}{\stackrel{\lower0.2ex\hbox{$\scriptscriptstyle   
             \mathbf{D}[0,1] $}}{\Rightarrow}}
\newcommand{\weak}{\stackrel{\lower0.2ex\hbox{$\scriptscriptstyle
                    \it{w} $}}{\rightarrow}}
\newcommand{\dist}{\stackrel{\lower0.2ex\hbox{$\scriptscriptstyle
                    \mathcal{D} $}}{\rightarrow}}                    
\newcommand{\tauconv}{\stackrel{\lower0.2ex\hbox{$\scriptscriptstyle
                    \it{\tau} $}}{\rightarrow}}
\newcommand{\calX}{\mathcal{X}}

\def\R{\mathbb{R}}

\def\Prob{\mathbb{P}}

\def\E{\mathbb{E}}

\def\calE{\mathcal{E}}

\def\frakF{\mathfrak{F}}
\def\linf{\ell _{\infty}(\mathfrak{F})}
\def\lin{\ell _{\infty}}

\newcommand{\M}{\mathcal{M}}
\newcommand{\re}{\mathcal{H}}

\newcommand{\norm}[1]{\left|\left|#1\right|\right|}
\newcommand{\abs}[1]{\left|#1\right|}

\theoremstyle{theorem}
\newtheorem{theorem}{Theorem}[section]
\newtheorem{lemma}[theorem]{Lemma}

\newtheorem{corollary}[theorem]{Corollary}

\theoremstyle{definition}

\newtheorem{remark}[theorem]{Remark}

\begin{document}

\title[Empirical processes in importance sampling]{Moderate deviations for importance sampling estimators of risk measures}        



\author[P.~Nyquist]{Pierre Nyquist}
\address[P.~Nyquist]{Department of Mathematics, KTH, 100 44
Stockholm, Sweden}\email{{pierren@kth.se}}  
\date{\today}




\begin{abstract} 
Importance sampling has become an important tool for the computation of tail-based risk measures. Since such quantities are often determined mainly by rare events standard Monte Carlo can be inefficient and importance sampling provides a way to speed up computations. This paper considers moderate deviations for the weighted empirical process, the process analogue of the weighted empirical measure, arising in importance sampling. The moderate deviation principle is established as an extension of existing results. Using a delta method for large deviations established by Gao and Zhao (\emph{Ann. Statist.}, 2011) together with classical large deviation techniques, the moderate deviation principle for the weighted empirical process is extended to functionals of the weighted empirical process which correspond to risk measures. 
The main results are moderate deviation principles for importance sampling estimators of the quantile function of a distribution and Expected Shortfall. 
\end{abstract}

%


\maketitle   
\section{Introduction}
Importance sampling has become a popular tool for making Monte Carlo simulation more efficient. In particular when used to estimate quantities largely determined by rare events. An importance sampling algorithm is defined in terms of a change of measure from the original dynamics of the system under consideration. The idea is that the important events for the quantity one is trying to estimate will occur more frequently under the new dynamics. The error introduced by using a different dynamics when sampling is corrected for by associating with each sample a weight corresponding to likelihood ratio associated with the change of measure. Since many such changes of measure are possible the question of which one is more efficient becomes imperative for choosing simulation algorithm. 


In the financial and actuarial context, risk measures such as Value-at-Risk and Expected Shortfall are commonly used to quantify risk. These, and other risk measures, depend on the tail of the loss distribution and for all but very simple models exact formulas are not available. Therefore, stochastic simulation, in particular Monte Carlo methods, is emerging as an indispensable tool for computing such quantities. Many risk measures can be formulated as functionals of the loss distribution. When these functionals are mainly determined by rare events - in this setting, events far out in the tail of the distribution - the computational cost of standard Monte Carlo can be too high for practical use. In order to reduce the computational cost, while maintaining a desired accuracy, importance sampling is a viable alternative. For estimating probabilities or, more generally, expectations, the efficiency of an algorithm is expressed in terms of the variance of the resulting estimator. A great amount of work has gone into studying such problems. One successful approach involves relating the estimation problem to a stochastic game and studying subsolutions of the accompanying Isaacs equation, see for example \cite{DupuisWang, DupuisWangSubsol} for some of the early work in this area. Another avenue for analyzing efficiency is provided by so-called Lyapunov inequalities \cite{BlanchetLiu, BlanchetGlynn, BlanchetGlynnLeder}. However, the amount of work that has gone into studying computation of quantiles and other risk measures is far less. Notable exceptions are for example \cite{Glasserman2002} and \cite{Glynn96} in which importance sampling estimation of a quantile is studied. 
Since risk measures are (often) non-linear functionals of a distribution, estimators are typically biased. Therefore, variance is no longer the canonical measure on which to base efficiency analysis of simulation algorithms. 

In \cite{Hult2011} efficiency of importance sampling algorithms is studied from the perspective of empirical processes. The authors establish central limit theorems for the empirical processes that correspond to importance sampling estimators of certain risk measures. The risk measures under consideration are Value-at-Risk and Expected Shortfall. As a measure of efficiency the authors study, in the rare event limit, the variance of the limiting random element in the central limit theorem. 
The main tools of \cite{Hult2011} are a central limit theorem for the empirical process corresponding to the underlying weighted empirical measure and the delta method (see for example \cite{Wellner2000}).

This paper complements the central limit theorems in \cite{Hult2011} by studying the same weighted empirical process from the large deviation perspective. Such results are commonly known as moderate deviations and they describe the asymptotic of probabilities on an intermediate scale between the central limit theorem and the large deviation principle. Moderate deviation results add to the central limit theorem in that they describe the rate of convergence and provide insight into how asymptotic confidence intervals can be constructed. Therefore, the study of moderate deviation properties has become a standard problem when considering statistical estimators.

The main results of this paper are moderate deviation principles for the importance sampling estimator of a quantile function and Expected Shortfall, respectively. As a first result, based on \cite{Wu94} and \cite{Ledoux92} we obtain the large deviation principle (or, with a different name, moderate deviation principle) for the empirical process that corresponds to the weighted empirical measure arising in importance sampling. This is a rather straightforward extension of the results in \cite{Wu94} and \cite{Ledoux92} to the setting of weighted empirical measures. Using this extension, the moderate deviation principle is shown to hold for importance sampling estimators of a quantile function and Expected Shortfall. The main tool, aside from the moderate deviation principle for the weighted empirical process, is a delta method for large deviations established in \cite{GaoZhao2011}. However, for the Expected Shortfall the delta method is not enough and we need to consider the asymptotics of the estimators in more detail and make use of exponential approximations to obtain the moderate deviation principle.
The paper is aimed at establishing the relevant moderate deviation results for general importance sampling algorithms and apply them in the context of quantiles and Expected Shortfall. Concrete examples on efficiency analysis of specific algorithms will be reported on elsewhere.

Moderate deviations of empirical processes is a rather well-studied subject. Some general references are \cite{ArconesEmp, BorovkovMog2, deAcosta97, deAcostaChen98, Ledoux92, Wu94}. See also \cite{GaoZhao2011} and the references therein. For estimation using standard Monte Carlo there are a number of moderate deviations results. For example, in addition to establishing the delta method for large deviations, \cite{GaoZhao2011} studies moderate deviations for several common estimators. In \cite{GaoWang2011} the authors study the asymptotic behavior of Expected Shortfall and one of their results is the moderate deviation principle. For stochastic simulation methods other than standard Monte Carlo, the literature on moderate deviations seems to be more scarce. For quite some time, mean field interacting particle models have been studied extensively in connection with stochastic simulation. In this context, a pioneering work is \cite{DoucEtAl2005} which studies moderate deviations for particle filtering. Recently, \cite{DelMoralEtAl2012} investigated moderate deviations for a large class of interacting particle models of mean field type. 
Large deviation results for the weighted empirical measures arising in importance sampling, with applications to efficiency analysis of importance sampling algorithms, are obtained in \cite{HultNyq}. This paper can be seen as an empirical process analogue of that work.

The rest of the paper is organized as follows. Section \ref{sec:preliminaries} introduces the notation used in the paper and the necessary background on empirical processes and large deviations is presented. In Section \ref{sec:empIS} the connection between importance sampling and empirical processes is discussed and the moderate deviation principle is shown for the weighted empirical process that arises in importance sampling. This result is used in Section \ref{sec:MDPrisk} to obtain the moderate deviation principle for importance sampling estimators of the quantile function and Expected Shortfall. The proofs of some auxiliary results, used for the results in Section \ref{sec:MDPrisk}, are given in Section \ref{sec:aux}.

\section{Preliminaries}
\label{sec:preliminaries}
\subsection{Notation}
Throughout the paper $(E, \calE)$ denotes a measurable space. $\M_1 = \M_1 (E)$ and $\M_b = \M _b(E)$ denotes the space of probability measures on $E$ and the space of  signed measures of finite variation on $E$, respectively. For $\nu \in \M_b$ denote by $\M_b ^{\nu,0}$ the subset of measures $\eta \in \M_b$ such that $\eta \ll \nu$ and $\eta (E) = 0$. For any measure $\eta$ on $(E, \calE)$ and $p \geq 1$, $L_p(E,\eta)$ is the space of measurable real-valued functions $f$ such that $(\int \abs{f}^p\, d\eta )^{1/p} < \infty$. For a function $f \colon E \mapsto \R$, $\text{supp}(f)$ denotes the support of $f$ and similarly for measures. When $\frakF$ is a  collection of functions, $\text{supp}(\frakF)$ is the smallest measurable set such that $\text{supp}(f) \subset \text{supp}(\frakF)$ for every $f \in \frakF$. In particular, $\text{supp}(\frakF) = \cup _{f\in \frakF} \text{supp}(f)$ if the union is measurable. For a measurable set $A$, let $A^{o}$ and $\bar{A}$ denote interior and closure of $A$, respectively.
Let $\{ \lambda _n \}$ be an increasing sequence such that $\lambda_ n \rightarrow \infty$ and $\lambda_n = o (\sqrt{n})$ as $n \rightarrow \infty$. For two real-valued functions $f$ and $g$, $f = o(g)$ and $f \sim g$ means that $f(x)/g(x)$ tends to $0$ and $1$ respectively as $x \rightarrow \infty$; similarly for sequences. Let $ b_n =  \sqrt{n}/\lambda_n$. By the properties of the sequence $\{ \lambda_n \}$, $b_n \rightarrow \infty$ as $n$ grows. Throughout, the terms large deviation and moderate deviation are used interchangeably with the interpretation that the moderate deviation principle is the large deviation principle with certain speed and scaling. Technical results, although repeatedly referred to in terms of the moderate deviation principle, are stated as large deviation principles.

\subsection{Empirical processes}
Let $X_1,X_2,...$, be independent and identically distributed random variables taking values in $E$ according to the law $\mu \in \M _1$. For $n \geq 1$, the empirical measure corresponding to the $n$ first random variables is
$$ \mu_n = \frac{1}{n}\sum _{i=1} ^n \delta _{X_i},$$
where $\delta _x$ denotes a unit point mass at $x$. 
Monte Carlo estimation of quantities related to $\mu$ is based on this sequence of empirical measures. For example, the mean of a function $f$ under $\mu$ is estimated by $\mu _n(f)$. By the law of large numbers, as $n\rightarrow \infty$, the empirical measure $\mu _n$ converges in the weak topology to $\mu$ with probability 1.

Let $\frakF$ be a class of measurable functions $f$ such that $f \in L_1(E,\mu)$ and, for each $x \in E$, $\sup \{ \abs{f(x)} \colon f \in \frakF \}  < \infty$. Furthermore, let $\linf$ be the space of all bounded functions $F \colon \frakF \mapsto \R$. This space is henceforth equipped with the sup-norm, $ \norm{F}_{\frakF} = \sup \{  \abs{F(f)} \colon f \in \frakF \}$ for $F \in \linf$. There will be a slight abuse of notation in that, for any $x\in E$, $\norm{f(x)} _{\frakF} = \sup \{\abs{f(x)} \colon f \in \frakF \} $. 
For each finite measure $\eta$ on $(E,\calE)$ there is a corresponding element $\eta ^{\frakF} \in \linf$ defined by
$$ \eta ^{\frakF} (f) = \eta (f) = \int _E f d\eta, \ f \in \frakF.$$ 
In particular, there is an element $\mu_n ^{\frakF} \in \linf$ corresponding to the empirical measure. To ease notation the superscript is dropped whenever the context is clear. If for all $x \in E$, $\sup _{f \in \frakF} | f(x) - \mu(f) | < \infty$, then the empirical process $\xi _n$ given by
$$ \xi _n (f) = \sqrt{n} \big ( \mu _n (f) - \mu(f) \big ), \ f \in \frakF,$$
can be viewed as a map into $\linf$. 

In order to keep the discussion of measurability issues to a minimum, large deviation results for empirical processes are in general stated in terms of outer and inner probabilities (defined below). In Section \ref{sec:MDPrisk}, when considering estimation of the tail of a distribution we equip $\linf$ with the $\sigma$-algebra generated by all balls and coordinate projections. This is consistent with the approach taken in \cite{GaoZhao2011} and ensures the necessary measurability properties (see \cite{Wellner2000}, Section 1.7). Aside from this, no explicit assumptions are made regarding measurability of the random variables or the $\sigma$-algebras involved and the remaining results are to be interpreted in terms of outer and inner probabilities. In Theorems \ref{thm:Wu_thm2}, \ref{thm:1} and \ref{thm:2}, for a general class $\frakF$, separability of the processes is assumed. This is only to ensure sufficient measurability and for the application to risk measures the mentioned use of a specific $\sigma$-algebra takes care of this. See \cite{Dudley99, LeTa2010, Wellner2000} and the references therein for more details on measurability issues in connection with empirical processes and Banach space valued random variables in general. 
\subsection{Large deviations}
This paper is concerned with the large deviation principle for certain empirical processes. In order to introduce this concept, the notions of \emph{outer} and \emph{inner integral} and \emph{outer} and \emph{inner probability}, as defined in \cite{Wellner2000}, Section 1.2, must first be introduced. Let $(\Omega, \mathcal{F}, \Prob)$ be an arbitrary probability space, $\E$ the expectation operator associated with $\Prob$ and $T \colon \Omega \mapsto [-\infty, \infty] $ an arbitrary map. The \emph{outer integral} of $T$ with respect to $P$ is
$$Ê\E ^{\ast} [T] = \inf \{ \E [U] \colon U \geq T, \ U \colon \Omega \mapsto [-\infty, \infty] \ \textrm{is measurable and } \E[U] \ \textrm{exists}  \}.$$
The \emph{outer probability} of an arbitrary subset $B$ of $\Omega$ is defined as
$$Ê\Prob ^{\ast} (B) = \inf \{ P(A) \colon A \supset B , \ A \in \mathcal{F} \}.$$
The definitions of \emph{inner integral}, $\E _{\ast}[T]$, and \emph{inner probability}, $\Prob _{\ast} (B)$, are analogous to these with the obvious changes; see \cite{Wellner2000}, Section 1.2 and beyond. 

For some metric space $\mathcal{X}$, consider a sequence $(\Omega _n , \mathcal{F} _n, \Prob _n)$ of probability spaces and maps $X_n \colon \Omega _n \mapsto \mathcal{X}$. Suppose that $I \colon \mathcal{X} \mapsto [0, \infty]$ is lower semicontinuous with compact level sets. Then, $\{ X_n \}$ is said to satisfy the large deviation principle (LDP) in $\mathcal{X}$, with speed $c_n ^{-1}$, where $\{ c_n \} \subset \mathbb{R} _ +$ and $c_n \rightarrow \infty$, and rate function $I$ if, for any measurable set $A \subset \mathcal{X}$,
\begin{align*}
	- \inf _{x \in A^o} I(x) &\leq \liminf _{n \rightarrow \infty }\frac{1}{c_n} \log \Prob _{n \ast} (X_n \in A) \\
	&\leq \limsup _{n \rightarrow \infty} \frac{1}{c_n} \log \Prob _{n} ^{ \ast} (X_n \in A) \leq - \inf _{x \in \bar{A}} I(x).
\end{align*}
In the absence of any measurability issues, $\Prob _{n \ast}$ and $\Prob _{n} ^{\ast}$ are replaced by $\Prob _n$.

Sanov's theorem (\cite{Dembo98}, Theorem 6.2.10) states that the sequence $\{ \mu_n \}$ of empirical measures satisfies the LDP in $\M _1$, equipped with the $\tau$-topology, with speed $n^{-1}$ and rate function given by the relative entropy $\re (\cdot \mid \mu)$. As mentioned in \cite{Wu94} it holds that the sequence
$$ \Big \{ \frac{1}{\lambda (n)} \xi _n \Big \}= \Big\{ b_n\big( \mu_n - \mu \big) \Big\},$$
satisfies the LDP in $\M _b$, equipped with the $\tau$-topology, with speed $\lambda _n ^{-2}$ and rate function
\begin{equation*}
\label{eq:Rate_MD}
 I _{\mu}(\eta) = \begin{cases} \frac{1}{2}\int \Big ( \frac{d\eta}{d\mu}\Big)^2 d\mu, & \text{if $\eta \in \M _b ^{\mu, 0}$}, \\
	+ \infty, & \text{otherwise}.
\end{cases}
\end{equation*}
The first key result for what will follow is established in \cite{Wu94} and concerns the LDP for empirical processes based on a sequence of independent and identically distributed random variables. In the context of empirical processes the LDP is also referred to as the moderate deviation principle (MDP). Let $d_2 \colon \frakF \times \frakF \mapsto \R$ denote the pseudometric on $\frakF$ given by
\begin{equation*}
	d_2(f,g) = \Big ( \int (f-g) ^2 d\mu \Big ) ^{\frac{1}{2}}, \ f,g\in \frakF.
\end{equation*}
It is to be understood that if the reference measure $\mu$ is changed then the definitions of $I_{\mu}$ and $d_2$ are changed accordingly.
\begin{theorem}[cf.\ \cite{Wu94}, Theorem 5]
\label{thm:Wu_thm2}
Suppose that $\frakF$ is a class of functions in $L_2(E, \mu)$ and there exist constants $A \geq 1$ and $\delta \in (0,1)$ such that
for all integers $n,k \geq 1$,
\begin{equation}
\label{eq:cond}
Ê\lambda _{nk} \leq A k ^{-(\delta - 1/2)} \lambda _n.
\end{equation}
Then, $\{ (b_n (\mu _n - \mu)\} $ satisfies the LDP in $\linf$ with speed $\lambda _n ^{-2}$ and rate function
\begin{equation}
\label{eq:RateMb}
	I_{\frakF} (G) = \inf \Big \{ I_{\mu} (\eta) \colon \eta \in \M _b \ \text{and } \eta ^{\frakF} = G \ \text{on } \frakF \Big \}, \ G \in \linf,
\end{equation}
if and only if the following three conditions are fulfilled: \\

(i) $(\frakF , d_2)$ is totally bounded,

(ii) $b_n (\mu _n - \mu) \rightarrow 0$ in probability in $\linf$,

(iii) there exists $M > 0$ such that, for all $u >0$,
\begin{equation}
\label{eq:condTail}
Ê\limsup _{n \rightarrow \infty} \frac{1}{\lambda_n ^2} \log \bigl ( n \mu  (\norm{f(x)}_{\frakF} > u \lambda _n \sqrt{n}) \bigr ) \leq - \frac{u^2}{M}.
\end{equation}
\end{theorem}
\begin{remark}
The difference between the above statement and that of Theorem 5 in \cite{Wu94} is the formulation of conditions \eqref{eq:cond} and \eqref{eq:condTail}. Following closely \cite{Ledoux92}, which provides the main argument for the proof, the necessary and sufficient conditions for establishing the LDP are indeed \eqref{eq:cond} and \eqref{eq:condTail} and not those provided in \cite{Wu94} (conditions (3.5) and $(iii)$ of Theorem 5).
\end{remark}
The first new result of this paper, Theorem \ref{thm:1}, is a rather straightforward extension of Theorem \ref{thm:Wu_thm2} and arguments in \cite{Ledoux92} to the setting of weighted empirical measures. In Theorem 2 in \cite{Wu94} it is shown  that when the class $\frakF$ is uniformly bounded between $0$ and $1$ the restrictions on $\lambda _n$ and the condition $(iii)$ are not necessary (the latter is of course then trivially fulfilled). This enables us to drop the conditions on $\lambda _n$ for certain importance sampling algorithms, as described in Theorem \ref{thm:2}.
Note that the rate function \eqref{eq:RateMb} is precisely the rate function one would guess, in the sense that whenever the contraction principle is applicable (for example when $\frakF$ is finite), \eqref{eq:RateMb} is the obtained rate function.

The second result that we utilize is a delta method for large deviations established in \cite{GaoZhao2011}. Only the first part of the result is presented and the reader is referred to the original paper for the entire statement. Here, $(\Omega _n, \mathcal{F} _n , \Prob _n)$ is a sequence of probability spaces. 
\begin{theorem}[Theorem 3.1, \cite{GaoZhao2011}]
\label{thm:GaoZhao}
	Let $\calX$ and $\mathcal{Y}$ be two metrizable topological spaces. Let $\Phi \colon \mathcal{D}_{\Phi} \subset \calX \mapsto \mathcal{Y}$ be Hadamard-differentiable at $\theta$ tangentially to $\mathcal{D}_0 \subset \calX$. Let $X_n \colon \Omega _n \mapsto \mathcal{D}_{\Phi}, n\geq 1$, be a sequence of maps and let $\{ r_n \}$ be a sequence of positive real numbers satisfying $r_n \rightarrow \infty$. If $\{ r_n (X_n - \theta) \} $ satisfies the LDP with speed $c_n ^{-1}$ and rate function $I$, such that $\{ I < \infty \} \subset \mathcal{D}_0$, then $\{ r_n (\Phi(X_n) - \Phi(\theta)) \} $ satisfies the LDP with speed $ c_n ^{-1}$ and rate function
$$ I _{\Phi_{\theta} '}  (y) = \inf \{ I(x) \colon \Phi _{\theta} ' (x) = y \}.$$
\end{theorem}
\section{Moderate deviations for weighted empirical processes}
\label{sec:empIS}
In importance sampling one considers sampling from a distribution $\nu$ such that sampling from $\nu$ is, hopefully, beneficial for the estimation task at hand. In order for a distribution $\nu$ to be feasible it must hold that $\mu \ll \nu$. The weighted empirical measure corresponding to importance sampling is 
$$ \nu_{n} ^w = \frac{1}{n} \sum _{i=1} ^n w(X_i) \delta _{X_i},$$
where the $X_i$'s are independent random variables with common distribution $\nu$ and $w(\cdot)$ is the \emph{weight function}, given by the Radon-Nikodym derivative $d\mu / d\nu$. The (standard) empirical measure based on $X_1,...,X_n$ is denoted by $\nu _n$. The idea behind importance sampling is that if $\nu _n ^w$ is a good approximation of $\mu$, using $\nu _n ^w$ should give good estimation of the quantity of interest. However, depending on the quantity one is trying to estimate it may suffice to have a good approximation on a subset of $E$. In contrast to standard Monte Carlo, one cannot hope for the weighted empirical measure $\nu _n ^w$ to be, in some sense, close to $\mu$ over the entire space $E$. In \cite{HultNyq} this is taken into account by introducing a so-called \emph{importance function}, here denoted by $f_i$. The function $f_i$ is non-negative, measurable and $\mu (f_i) < \infty$ and the importance of different regions of $E$ is indicated by $f_i$. Then, it suffices to have $\mu \ll \nu$ on $\textrm{supp}(f_i)$ and it is possible to define the weight function $w = (d\mu / d\nu) I\{f > 0\}$ which is now well-defined over the entire space $E$. The interpretation of a good approximation is that $\nu _n ^{wf}$ is close to the weighted measure $\mu ^f$. For the remainder of the paper we do not make any explicit comments related to the importance function and it is considered included in the weight function $w$; for more discussion see \cite{HultNyq}.

Similar to how the empirical process $\xi _n$ is associated with the standard empirical measure there is an empirical process $\xi _{n} ^w$ associated with the weighted empirical measure $\nu _n ^w$,
\begin{equation*}
 \xi _{n} ^w (f) = \sqrt{n} \big ( \nu _n ^w (f) - \mu(f) \big ), \ f\in \frakF. 
 \end{equation*}
The first result yields the LDP for $\{ \xi _{n} ^w \}$ . 
\begin{theorem}
\label{thm:1}
	Let $\frakF$ be a class of real valued functions such that $0\leq f \leq 1$ for every $f \in \frakF$. 
		Suppose that $\E _{\nu} [\exp\{\alpha w(X) \}] < \infty$ for every $\alpha > 0$ and $\lambda _n$ satisfies \eqref{eq:cond}. If $\frakF$ is $\nu$-Donsker, then the empirical process sequence $ \{ b_n (\nu_{n} ^w - \mu\} $ corresponding to importance sampling satisfies the LDP in $\linf$ with speed $\lambda _n^{-2}$ and rate function 
$I ^w _{\frakF} \colon \linf \mapsto [0,\infty]$ given by
\begin{equation}
\begin{split}
\label{eq:rateXi}
	I ^w _{\frakF} (G) = \inf \Big \{ I_{\nu}(\eta) \colon \eta \in \M _b ^{\nu, 0}, \eta(wf) = G(f) \ \forall f \in \frakF \Big \}.
\end{split}
\end{equation}
Moreover, $ \{ b_n (\nu _{n} ^w - \mu) \} $ converges in probability to the zero element in $\linf$.
\end{theorem}
%
\begin{proof}
Let $w\frakF = \{ w f \colon f \in \frakF\}$. A key observation is that for every $f \in \frakF$, $\nu _{n} ^w (f) = \nu _n (wf)$ and $\mu(f) = \nu (wf)$. Therefore, obtaining the LDP for $\{ b_n ( \nu _{n} ^w - \mu )\}$ in $\linf$ is identical to obtaining the LDP for $ \{ b_n(\nu_n - \nu) \}$ in $\lin(w\frakF)$. Indeed, consider the mapping that takes $F \in \lin (w \frakF)$ to $F^w \in \linf$, where $F^w(f) = F(wf)$. This mapping is continuous and the contraction principle can be applied. Hence, it suffices to show the LDP for the non-weighted empirical process $\{ b_n(\nu _n - \nu) \}$ in $\lin (w \frakF)$. Moreover, the rate function implied by the contraction principle is precisely \eqref{eq:rateXi}.


The class $\frakF$ is by assumption $\nu$-Donsker. The assumptions of a uniform bound on $\frakF$, $\E _{\nu} [\textrm{exp}\{\alpha w(X) \}] < \infty$ for every $\alpha > 0$ and the permanence of the Donsker property \cite[Section 2.10]{Wellner2000} imply that $w\frakF$ is $\nu$-Donsker as well. It follows from Theorem 14.6 in \cite{LeTa2010} that $(w\frakF, d_2)$ is totally bounded. 
The Donsker property gives the (uniform) central limit theorem for $\nu _n - \nu$ and hence weak convergence of the sequence $\{ \sqrt{n} (\nu _n - \nu) \}$. This implies that the sequence $\{ b_n (\nu _n - \nu) \}$ converges to $0$ in probability; condition $(ii)$ of Theorem \ref{thm:Wu_thm2} holds. Remains to check \eqref{eq:condTail}. First, note that under the assumption on $\frakF$,
$$ \E _{\nu} [\textrm{exp} \{ \alpha \norm{f(X)} _{w \frakF} \}] \leq \E _{\nu} [\textrm{exp} \{ \alpha w(X) \} ].$$
Thus, under the assumption that $w(X)$ has finite exponential moments of all orders, condition $(iii)$ of Theorem \ref{thm:Wu_thm2} is satisfied; see comment right after the statement of the main theorem in \cite{Ledoux92}. Hence, by Theorem \ref{thm:Wu_thm2}, $\{ b_n (\nu _n - \nu ) \} $ satisfies the LDP in $\lin (w \frakF)$. This completes the proof.
\end{proof}
\begin{remark}
For different sequences $\{ \lambda _n \}$ the assumption on the exponential moments of $\norm{w(X)f(X)} _{\frakF}$ can be relaxed. For example, if $\lambda _n$ is on the form 
$\lambda _n = n ^{1/p - 1/2}$ for $p \in (1,2)$, it is enough to have $\E _{\nu} [\textrm{exp} \{ \alpha \norm{w(X) f(X)} _{\frakF} ^{2-p}] < \infty$ for every $\alpha > 0$; see Corollary 1 in \cite{Ledoux92}. 
\end{remark}
When the weight function $w$ is bounded on $\textrm{supp}( {\frakF})$ the restrictions on $\{ \lambda _n \}$ can be dropped. This is similar to Theorem 2 in \cite{Wu94}.  
\begin{theorem}
\label{thm:2}
	Let $\frakF$ be a class of real valued functions such that $0\leq f \leq 1$ for every $f \in \frakF$. Suppose that $w$ is bounded on $\text{supp}(\frakF)$. If $\frakF$ is $\nu$-Donsker, then the empirical process sequence $ \{ b_n ( \nu _n ^w - \mu ) \} $ corresponding to importance sampling satisfies the LDP in $\linf$ with speed $\lambda _n ^{-2}$ and rate function given by \eqref{eq:rateXi}. Moreover, $ \{ b_n ( \nu _n ^w - \mu ) \} $ converges in probability to the zero element in $\linf$.
\end{theorem}
Theorem \ref{thm:2} can be shown by modifying the proof of Theorem 2 in \cite{Wu94} in a direct way using the the assumed bound on $w$. The details are omitted.

%
%
%
%
%
\section{Moderate deviations for importance sampling estimators of risk measures}
\label{sec:MDPrisk}
Henceforth, let the underlying space be $E = \R$ and assume that $\{ \lambda _n \}$ satisfies \eqref{eq:cond}. Note that due to Theorem \ref{thm:2} the results of this section will hold without this assumption if $w$ is bounded on $\textrm{supp}(\frakF)$. Let $F$ denote the distribution function of $\mu$, $F(t) = \mu (I\{ \cdot \leq t \})$, and let $T$ be the 
\emph{tail} of $F$, $T(t) = 1- F(t)$. Denote by $T_{n} ^w$ the importance sampling estimator of $T$,
$$ T_{n} ^w (t) = \nu _{n} ^w (I\{\cdot > t\}) , \ t\in \R.$$
Using Theorem \ref{thm:1} a moderate deviation result is easily proved for the importance sampling estimator of the tail $T$. Recall from Section \ref{sec:preliminaries} that, to ensure the necessary measurability, for this first result the space $\linf$ is equipped with the $\sigma$-algebra generated by all balls and coordinate projections (see \cite{GaoZhao2011}). For the remaining results no such assumptions are made.
\begin{corollary}
\label{cor:tailIS}
	Let $\frakF _a = \{ I\{ \cdot > t\} \colon t \geq a \}$ for some $a > 0$. Suppose that $\nu$ satisfies the assumptions of Theorem \ref{thm:1}. Then, the sequence
$ \{ b_n (T_{n} ^w - T) \} $ satisfies the LDP in $\lin [a,\infty]$ with speed $\lambda _n^{-2}$ and rate function $I ^w _{\frakF _a}$ given by 
\begin{equation*}
\begin{split}
	I^w _{\frakF _a}(G) = \inf \Big \{ I_{\nu} (\eta) \colon \eta \in \M _b ^{\nu, 0}, \eta(I \{ \cdot > t\} w) = G(t) \ \forall t \in [a,\infty] \Big \},
\end{split}
\end{equation*}
where $G \in \lin [a,\infty]$. Moreover, the sequence converges in probability to the zero element in $\lin [a,\infty]$.
\end{corollary}
\begin{proof} The proof follows from Theorem \ref{thm:1} since the class $\frakF _a$ is $\nu$-Donsker for any probability measure $\nu$.
\end{proof}

With the uniform moderate deviation principle established for importance sampling related to the tail of the original distribution, the corresponding results for estimators of certain functionals can be obtained through the delta method (Theorem \ref{thm:GaoZhao}). We illustrate this by considering importance sampling estimation of the quantile function corresponding to $\mu$. For a non-increasing c\`{a}dl\`{a}g function $H \colon \R \mapsto \R$, define the inverse map
\begin{equation*}
	\Phi_p(H) = H^{-1}(p) = \inf \{u \colon H(u) \leq p \}, \ p\in (0,1).
\end{equation*}
The importance sampling estimator of the true quantile function $T^{-1}$ is denoted by $(T_{n} ^w ) ^{-1}$.
\begin{theorem}
\label{thm:quantileIS}
Suppose $F$ has continuous density $f>0$ with respect to Lebesgue measure on the interval $[T^{-1}(p)-\epsilon, T^{-1}(q) + \epsilon]$ for $0<q<p<1$ and some $\epsilon > 0$. If the sampling distribution $\nu$ satisfies the assumptions of Theorem \ref{thm:1}, then the sequence
\begin{equation*}
\label{eq:seqQuantile}
\bigl \{ b_n \bigl( (T_n ^w )^{-1} - T^{-1} \bigr) \bigr \}, 
 \end{equation*}
satisfies the LDP in $\lin [q,p]$ with speed $\lambda _n ^{-2}$ and rate function
\begin{align}
\label{eq:rateQuant}
I ^w _{T^{-1}}(G) = \inf \Big \{ I_{\nu} (\eta) \colon \eta \in \M _b ^{\nu, 0}, \frac{\eta(I\{ \cdot > T^{-1}(u) \} w)}{f(T^{-1}(u))} = G(u) \  \forall u\in [q,p] \Big \}.
\end{align}	
Moreover, the sequence converges in probability to the zero element in $\lin [q,p]$.
\end{theorem}
\begin{proof} A simple modification of Corollary \ref{cor:tailIS} gives the LDP for $\{ b_n ( T_{n} ^w - T) \} $ on $\lin [a,b]$ for any $b> a$. 
Take $a = T^{-1}(p)-\epsilon$ and $b = T^{-1}(q) + \epsilon$ and let $\frakF _{a,b}$ be the corresponding collection of indicator functions. 
Lemmas 3.9.20 and 3.9.23 in \cite{Wellner2000} state that, when viewed as a map from the set of distribution functions restricted to $[a,b]$ into $\lin [q,p]$, the left-continuous version of the inverse map is Hadamard differentiable at $F$ tangentially to the set of continuous functions on $[a,b]$. 
Hadamard differentiability at $T$ holds in the same way for the right-continuos version of the inverse map, the one considered here. Since the derivative of the tail $T$ is $-f$ the corresponding derivative of the inverse is $\alpha \mapsto \frac{\alpha(T^{-1})}{f(T^{-1})}$. Hadamard differentiability together with the above modification of Corollary \ref{cor:tailIS} and Theorem \ref{thm:GaoZhao} imply that the quantile process $\{ b_n ((T_{n} ^w) ^{-1} - T^{-1})\}$ satisfies the LDP with speed $\lambda_n ^{-2}$ and rate function
\begin{align*}
\tilde{I} ^w _{T^{-1}}(G) &= \inf \Big \{ I^w _{\frakF _{a,b}}(\alpha) \colon \alpha \in \lin [a,b], \ \frac{\alpha(T^{-1}(u))}{f(T^{-1}(u))} = G(u) \  \forall u\in [q,p] \Big \},
\end{align*}
for $G \in \lin[q,p]$.
Next, we identify the rate function with $I^w _{T^{-1}}$ in \eqref{eq:rateQuant} by showing inequality in both directions. Take an arbitrary $G \in \lin [q,p]$. First, suppose that there is either no $\alpha \in \lin [a,b]$ such that $\alpha (T^{-1} (u)) = f(T^{-1}(u))G(u)$ for each $u \in [q,p]$ or, if such an $\alpha$ exists, no $\eta \in \M _b ^{\nu, 0}$ such that $\eta (I \{ \cdot \geq t \}w) = \alpha (t)$ for each $t \in [a,b]$. Then, both $\tilde{I}^w _{T^{-1}} (G)$ and $I ^w _{T^{-1}}$ are infinite and equality trivially holds. Therefore, we can assume that $G \in \lin [q,p]$ is such that that we can choose $\alpha ^* \in \lin [a,b]$ and $\eta^* \in \M _b ^{\nu, 0}$ such that $\alpha ^* (T^{-1}(u)) = G(u) f(T^{-1}(u))$ for each $u\in [q,p]$ and $\eta ^* (I \{ \cdot \geq t \}w) = \alpha ^*(t)$ for each $t \in [a,b]$. 
Clearly,
$$Ê\tilde{I} ^w _{T^{-1}} (G) \leq I ^w _{\frakF _{a,b}} (\alpha ^*) \leq I_{\nu}(\eta ^*).$$
The left-hand side has no dependence on $\eta^*$ and taking infimum yields
\begin{align*} 
\tilde{I} ^w _{T^{-1}} (G) & \leq \inf \Big \{ I_{\nu} (\eta) \colon \eta \in \M _b ^{\nu, 0}, \frac{ \eta (I\{ \cdot > T^{-1}(u) \}w)}{f(T^{-1}(u))} = G(u) \ \forall u\in [q,p] \Big \} \\
& = I ^w _{T^{-1}} (G).
\end{align*}
For the reverse inequality, note that for any $\delta >0$, there is $\alpha _{\delta} \in \lin [a,b]$ such that $\alpha _{\delta} (T^{-1} (u)) = f(T^{-1}(u))G(u)$ for each $u \in [q,p]$ and 
$$Ê\tilde{I} ^w _{T^{-1}}(G) + \delta \geq I ^w _{\frakF _{a,b}}(\alpha _{\delta}).$$
Similarly, there is $\eta _{\delta} \in \M _b ^{\nu, 0} $ such that $\eta _{\delta} (I\{ \cdot > t\}w) = \alpha _{\delta}(t)$ for each $t \in [a,b]$ and
$$ÊI ^w _{\frakF _{a}}(\alpha _{\delta}) + \delta \geq I_{\nu}(\eta _{\delta}).$$
Together these inequalities yield
$$Ê\tilde{I} ^w _{T^{-1}}(G) + 2\delta \geq I _{\nu} (\eta _{\delta}).$$
The above holds for any $\delta >0$ and since $G$ was arbitrary we conclude that $I ^w_{T^{-1}}$ as defined in \eqref{eq:rateQuant} is indeed the rate function for the quantile process. 
 
Finally, the convergence in probability follows from the modified version of Corollary \ref{cor:tailIS} and the (standard) delta method.
\end{proof}

Next, we study moderate deviations related to importance sampling estimation of the risk measure Expected Shortfall. For a non-increasing c\`{a}dl\`{a}g function $H$ and $0<p<1$, let $\gamma_{p}(H)$ be defined by
$$Ê\gamma _p (H) = \frac{1}{p} \int _0 ^p H(u)du.$$
If $T$ is the tail of the distribution of a random variable then $\gamma_p(T^{-1})$ is called the Expected Shortfall at level $p$ and an importance sampling estimator is given by $\gamma _p ((T_{n} ^w)^{-1})$.

For the remainder of the paper we will let $q$ tend to zero along some monotone sequence $\{ q_m \}$ such that for each $m$, $T(T^{-1}(q_m)) = q_m $. This is possible since there can be at most countably many points $q$ such that $T$ is not continuous at $T^{-1}(q)$. In order to obtain the LDP for the empirical process which corresponds to Expected Shortfall we make the following assumptions. 
\begin{itemize}
\item{ $\mu$ has finite second moment,
\begin{equation}
\label{eq:hyp1}
	\E_{\mu}[X^2] < \infty, \tag{A1}
\end{equation}
and, as $m \rightarrow \infty$,
\begin{equation}
\label{eq:hyp2}
	q_m ^2 = o (f(T^{-1}(q_m))).	\tag{A2}	
\end{equation}}
\item{ The sampling distribution $\nu$ has finite weighted second moment,
\begin{equation}
\label{eq:hyp3}
	\E_{\nu} [(X w(X))^2] = \E _{\mu} [X^2 w(X)] < \infty,	\tag{A3}
\end{equation}
and, as $m \rightarrow \infty$,
\begin{equation}
\label{eq:hyp4}
	q_m ^2 \E_{\nu} [w(X)^2 I \{ X > T^{-1}(q_m)\} ] = o (f (T^{-1}(q_m))^2). \tag{A4}
\end{equation}}
\end{itemize}
Before stating the main result on estimators of Expected Shortfall we take a more detailed look at \eqref{eq:hyp1}-\eqref{eq:hyp4}. Consider first a regularly varying tail $T$ with index $-\alpha$, $\alpha >0$. That is, for every $t>0$,
$$Ê\lim _{x \rightarrow \infty} \frac{T(tx)}{T(x)} = t^{-\alpha}. $$
We let $L$ denote a generic slowly varying function, that is a function such that $L(tx)/L(x) \rightarrow 1$ as $x \rightarrow \infty$ (slowly varying at $\infty$). For a thorough treatment of regular variation and the results used in what follows, see \cite{resnick2008}. In order for $\mu$ to have a finite second moment it is required that $\alpha >2$. Consider the assumption \eqref{eq:hyp2}. By Karamata's Theorem, $f(x) \sim \alpha x^{-1} T(x)$ as $x \rightarrow \infty$. Thus, 
	$$ f(T^{-1} (q_m)) \sim \frac{\alpha T(T^{-1}(q_m))}{T^{-1}(q_m)} = \frac{\alpha q_m}{T^{-1}(q_m)} \ \textrm{as } m \rightarrow \infty,$$
which leads to
	$$Ê\frac{f(T^{-1}(q_m))}{q_m ^2} \sim \frac{\alpha}{q_m T^{-1} (q_m)} \ \textrm{as } m \rightarrow \infty.$$
The denominator goes to zero because the distribution has finite first moment. Note that in the regularly varying case this can be readily seen from the fact that the quantile behaves like $q_m ^{- 1/\alpha} L(q_m)$, where $L$ is slowly varying at $0$. Therefore, the denominator behaves like $ q_m ^{1 - 1/\alpha} L(q_m)$ which goes to zero for $\alpha > 1$. When the underlying distribution is light-tailed, e.g. Gaussian, the decay is even more rapid. Hence, assumption \eqref{eq:hyp2} is satisfied for a large class of both light-tailed and heavy-tailed distributions.

The more involved assumption is \eqref{eq:hyp4} concerning the sampling distribution. Again, let the tail $T$ be regularly varying with index $- \alpha$. Using the same asymptotic equivalence as above, 
\begin{align*}
	\frac{f(T^{-1}(q_m))^2}{q_m ^2 \E _{\nu} [w(X) ^2 I \{ÊX > T^{-1} (q_m)\}]} &\sim \frac{\alpha ^2 q_m ^2}{(T^{-1}(q_m) )^2} \frac{1}{q_m ^2 \E _{\nu} [w(X) ^2 I \{ÊX > T^{-1} (q_m)\}]} \\
	&= \frac{\alpha ^2}{(T^{-1}(q_m) )^2 \E _{\nu} [w(X) ^2 I \{ÊX > T^{-1} (q_m)\}]}, 
\end{align*}
as $m \rightarrow \infty$. Moreover, since $T^{-1} (q_m) \sim q_m ^{- 1/\alpha} L(q_m)$ as $m \rightarrow \infty$,
\begin{align*}
	\frac{f(T^{-1}(q_m))^2}{q_m ^2 \E _{\nu} [w(X) ^2 I \{ÊX > T^{-1} (q_m)\}]} \sim \frac{\alpha ^2}{ q_m ^{-2 /\alpha} L(q_m) \E _{\nu} [w(X) ^2 I \{ÊX > T^{-1} (q_m)\}]}.
\end{align*}
For standard Monte Carlo ($w \equiv 1$), $\E _ {\nu} [w(X)^2 I \{ÊX > T^{-1} (q_m)\}] = q_m$ and 
\begin{align*}
	\frac{f(T^{-1}(q_m))^2}{q_m ^2 \E _{\nu} [w(X) ^2 I \{ÊX > T^{-1} (q_m)\}]} \sim q_m ^{2 /\alpha -1} L(q_m)^{-1},
\end{align*}
which goes to $ \infty$ if $\alpha > 2$. The inverse converges to $0$ as $m \rightarrow \infty$ and a completely analogous analysis for a light-tailed (Gaussian-like) setting yields the same result. This shows that even for standard Monte Carlo the assumption \eqref{eq:hyp4} is satisfied for a rather large class of distributions.

Let $\sigma _{q,p} ^2 (w)$ and $\sigma _p ^2 (w)$ be defined by
\begin{align*}
	p^2 \sigma _{q,p}^2 (w) & = ( T^{-1}(q) - T^{-1}(p))^2 \int _{T^{-1}(q)} ^{\infty} w(x) \nu(dx) \\ 
	& \quad + \int _{T^{-1}(p)} ^{T^{-1}(q)} (x - T^{-1}(p))^2 w(x)^2 \nu(dx) \\
	& \qquad - \Big( \int _{T^{-1}(p)} ^{T^{-1}(q)} (x- T^{-1}(p)) \mu (dx) + q (T^{-1}(q) - T^{-1}(p)) \Big ) ^2,
\end{align*}
and
\begin{align*}
	\sigma _p ^2 (w) = \frac{1}{p^2} \int _{T^{-1}(p)} ^{\infty} (x - T^{-1}(p))^2 w(x)^2 \nu(dx) - \frac{1}{p^2} \Big( \int _{T^{-1}(p)} ^{\infty} (x - T^{-1}(p)) \mu(dx) \Big)^2. 
\end{align*}
\begin{theorem}
\label{thm:ESIS}
	Assume \eqref{eq:hyp1}-\eqref{eq:hyp4} and that the hypotheses of Theorem \ref{thm:quantileIS} hold for each $q_m$. Then, the sequence
	\begin{equation*}
		\bigl \{ b_n \bigl( \gamma _p((T_{n} ^w)^{-1}) - \gamma _p (T^{-1}) \bigr ) \bigr \},
	\end{equation*}
	satisfies the LDP in $\R$ with speed $\lambda _n^{-2}$ and rate function
	\begin{align}
	\label{eq:rateES}
		I^w _{p} (z) = \frac{z^2}{2 \sigma_p ^2 (w)}, \ z \in \R.
	\end{align}
\end{theorem}
\begin{remark}
Since standard Monte Carlo corresponds to the special case $w \equiv 1$, Theorem \ref{thm:ESIS} establishes the moderate deviation principle for standard Monte Carlo estimation of Expected Shortfall. In \cite{GaoWang2011} the authors study various asymptotics of Expected Shortfall (there called conditional Value-at-Risk, CVaR) based on standard Monte Carlo estimation. Thus, we obtain as a special case of Theorem \ref{thm:ESIS} their result on the moderate deviation principle \cite[Theorem 1.3]{GaoWang2011}. In order to see that the two rate functions are indeed the same, note first that from Theorem \ref{thm:ESIS}, when $w \equiv 1$, the rate function is $I^w _p (z) = z^2 / (2 \sigma _p ^2 (1))$, with
\begin{align}
	\sigma _{p} ^2 (1) = \frac{1}{p^2} \int _{T^{-1}(p)} ^{\infty} (x - T^{-1}(p))^2 \mu(dx) - \frac{1}{p^2}\Big( \int _{T^{-1}(p)} ^{\infty} (x - T^{-1}(p)) \mu(dx) \Big)^2. 
\end{align}
In \cite{GaoWang2011} the Expected Shortfall is defined using the left-continuous inverse. Given that the inverse is continuous at $\alpha \in (0,1)$, the the left-continuous inverse evaluated at $\alpha$ is equal to the right-continuous inverse $T^{-1}$ evaluated at $1-\alpha$:
\begin{align*}
	\inf \{ t \colon F(t) \geq \alpha \} = \inf \{ t \colon 1-\alpha \geq 1-F(t) \} = \inf \{ t \colon 1-\alpha \geq T(t) \} = T^{-1}(1-\alpha).
\end{align*}
Hence, we identify that the Expected Shortfall at level $\alpha$ as considered in \cite{GaoWang2011} is equivalent to the Expected shortfall at level $p = 1-\alpha$ using the above definition.
To see that the two rate functions agree, use that the rate function of \cite{GaoWang2011} is expressed in terms of the second moment of the random variable $Z(\alpha)$, 
\begin{align*}
	Z(\alpha) = \frac{1}{1-\alpha} (X - T^{-1}(1-\alpha))^+ - \frac{1}{1-\alpha} \int _{T^{-1}(1-\alpha)} ^{\infty} T(x) dx,
\end{align*}
with $X$ distributed according to $\mu$ as before. Namely, the rate function can be written as
\begin{align*}
	I(z) = \frac{z^2}{2 \E [Z(\alpha) ^2]}.
\end{align*}
Using Fubini's theorem it is readily seen that
\begin{align*}
	\E[Z(\alpha) ^2] = \sigma _{1-\alpha} ^2 (1). 
\end{align*}
Thus, the rate function of \cite{GaoWang2011} for Expected Shortfall at level $\alpha$ agrees with the rate function of Theorem \ref{thm:ESIS} when $p = 1-\alpha$ and $w \equiv 1$.
\end{remark}
In general, the map $T \mapsto \gamma _p (T^{-1})$ need not be Hadamard differentiable due to the fact that the $q$:th quantile may blow up at $q=0$. It may be that the quantile map does not yield an element of $\lin [0,p]$ because the resulting functional is not bounded. Therefore, it is not just a matter of applying the delta method in order to obtain the LDP for the Expected Shortfall. Instead, we use what in large deviation analysis is known as an \emph{exponentially good approximation} (cf.\ \cite{Dembo98}, Definition 4.2.14): Let $(\mathcal{X},d)$ be a metric space and for each $\delta > 0$ let 
$$Ê\Gamma _{\delta} = \{Ê(x, y) : d(x,y) \geq \delta \} \subset \mathcal{X} \times \mathcal{X}.$$
For all $n,m \in \mathbb{Z} _+$, let $(\Omega, \mathcal{F} _n, \Prob _{n,m})$ be a probability space and let the $\mathcal{X}$-valued random variables $\tilde{X} _{n}$ and $X_{n, m}$ be distributed according to the joint law $\Prob_{n, m},$ with marginals $\tilde{\mu}_{n}$ and $\mu _{n, m}$ respectively. $\{ÊX_{n, m}\}$ is called an exponentially good approximation of $\{Ê\tilde{X} _{n} \}$ if, for every $\delta > 0 $, the set $\{ \omega : (\tilde{X} _{n}, X _{n, m}) \in \Gamma _{\delta } \}$ is $\mathcal{F} _{n}$-measurable and for each $K > 0$ there is a $m_K$ such that 
$$Ê\limsup _{n \rightarrow \infty} \frac{1}{\lambda_n ^2} \log \Prob_{n, m} (\Gamma _{\delta}) \leq - K,$$
for all $m \geq m_K$. Similarly, the sequence of measures $\{ \mu _{n, m} \}$ is an exponentially good approximation of $\{ \tilde{\mu} \}$ if one can construct probability spaces $\{ (\Omega, \mathcal{F}_{n}, \Prob_{n, m}) \}$ as above.

The following is an outline of the proof of Theorem \ref{thm:ESIS}. The idea is to consider a truncated version of the mapping $\gamma _p$. For general non-increasing c\`{a}dl\`{a}g functions $H$ and $0<q<p<1$, define the mapping $\gamma _{q,p} \colon \lin [q,p] \mapsto \R$ by
$$ \gamma_{q,p}(H) = \frac{1}{p}\int _q ^p H(u)du. $$
When $T^{-1} \in \lin [q,p]$ the mapping $\gamma _{q,p}$ is Hadamard differentiable at $T^{-1}$ and, by an application of Theorem \ref{thm:GaoZhao}, the LDP for the empirical process corresponding to $\gamma _{q,p}$ follows immediately (Lemma \ref{lemma:ESqpIS} below). Notice that as $m \rightarrow \infty$, $q_m$ becomes arbitrarily close to $0$ and it is reasonable to think that the random variables $ \gamma _{q_m,p} ((T_n ^w)^{-1}) - \gamma _{q_m, p} (T^{-1})$ should be close to $\gamma _p ((T_n ^w ) ^{-1}) - \gamma_p (T^{-1})$. The next step of the proof of Theorem \ref{thm:ESIS} is to show that this is indeed so in the sense of an exponentially good approximation. That is, that one can choose $m$ large enough so that the difference $\gamma _{q_m} ((T_n ^w)^{-1}) - \gamma _{ T_m} (F^{-1})$
can be made arbitrarily small in the sense described above. Lemma \ref{lemma:expApprox} establishes this exponentially good approximation using an upper bound derived in \cite{Hult2011} for the probability of an absolute error $\geq \delta$, together with some one-dimensional LDP's originating from Corollary \ref{cor:tailIS} and Theorem \ref{thm:quantileIS}. Theorem 4.2.16 in \cite{Dembo98} states that if a sequence of random variables satisfy the LDP, any other sequence for which it is an exponentially good approximation will also satisfy the LDP. Moreover, the rate function is expressed in terms of the rate function associated with the first sequence. Hence, once we establish that $\gamma _{q_m,p} ((T_n ^w)^{-1}) - \gamma _{q_m, p} (T^{-1})$ is indeed an exponentially good approximation of $\gamma _p ((T_n ^w ) ^{-1}) - \gamma_p (T^{-1})$, Theorem \ref{thm:ESIS} is proved.

\begin{lemma}
\label{lemma:ESqpIS}
	Assume that the hypotheses of Theorem \ref{thm:quantileIS} are satisfied. Then, the sequence
	\begin{equation*}
	\label{eq:seqESqp}
		 \bigl \{ b_n \bigl( \gamma_{p,q} ((T_{n} ^w)^{-1}) - \gamma_{p,q} (T^{-1}) \bigr) \bigr \}, 
	\end{equation*}
	satisfies the LDP in $\R$ with speed $\lambda _n^{-2}$ and rate function
	\begin{equation}
	\label{eq:rateESqp}
		I^w _{q,p} (z) = \frac{z^2}{2 \sigma ^2 _{q, p}(w)}, \ z \in \R.
	\end{equation}
\end{lemma}
\begin{lemma}
\label{lemma:expApprox}
	Under the assumptions of Theorem \ref{thm:ESIS}, the sequence \\ $\{ b_n( \gamma_{q_m,p} ((T_{n} ^w)^{-1}) - \gamma_{q_m,p} (T^{-1})) \}$ is an exponentially good approximation of  \\$\{ b_n ( \gamma_{p} ((T_{n} ^w)^{-1}) - \gamma_{p} (T^{-1})) \}$.
\end{lemma}
Lemmas \ref{lemma:ESqpIS} and \ref{lemma:expApprox} are proved in Section \ref{sec:aux}.

The proof of this sections main result, Theorem \ref{thm:ESIS}, is simply a matter of combining Lemmas \ref{lemma:ESqpIS} and \ref{lemma:expApprox}; the following is basically a concise reiteration of the outline above.
\begin{proof}[Proof of Theorem \ref{thm:ESIS}]
	By Lemma \ref{lemma:expApprox} $\{ b_n( \gamma_{q_m,p} ((T_{n} ^w)^{-1}) - \gamma_{q_m,p} (T^{-1})) \}$ is an exponentially good approximation of  $\{ b_n ( \gamma_{p} ((T_{n} ^w)^{-1}) - \gamma_{p} (T^{-1})) \}$. By the large deviation principle of Lemma \ref{lemma:ESqpIS} and Theorem 4.2.16 in \cite{Dembo98} it follows that the sequence $\{ b_n ( \gamma_{p} ((T_{n} ^w)^{-1}) - \gamma_{p} (T^{-1})) \}$ also satisfies the LDP in $\R$ with speed $\lambda_n ^{-2}$. The associated rate function is
\begin{align*}
	I^w _{p} (z) = \sup _{\delta > 0} \liminf _{m\rightarrow \infty} \inf _{y \in B_{z,\delta}} I^w _{q_m,p}(y), 
\end{align*}
where $B_{z,\delta} = \{ x \in \R : |x-z| < \delta \}$. 
The rate function $I^w _{q_m,p}$ associated with the truncated sequence is given by \eqref{eq:rateESqp} and it follows that
\begin{align*}
	I^w _{p} (z) &= \sup _{\delta > 0} \liminf _{m\rightarrow \infty} \inf _{y \in B_{z,\delta}} \frac{y^2}{2 \sigma ^2 _{q_m,p}(w)} \\
	& = \frac{1}{2} \frac{\bigl (\sup _{\delta >0} \inf _{y \in B_{z,\delta}} y^2 \bigr)}{\limsup _{m \rightarrow \infty} \sigma ^2 _{q_m,p}(w)} \\
	&= \frac{z^2}{2} (\limsup _{m \rightarrow \infty} \sigma ^2 _{q_m,p}(w) )^{-1}.
\end{align*}
It is easily checked that under the assumptions \eqref{eq:hyp1}-\eqref{eq:hyp4},
\begin{align*}
	\limsup _{m \rightarrow \infty} \sigma ^2 _{q_m,p}(w) = \sigma _p ^2 (w),
\end{align*}
and the rate function is indeed
\begin{align*}
	I^w _p (z) = \frac{z^2}{2 \sigma _p ^2 (w)}.
\end{align*}

%
%
%
\end{proof}
\section{Proof of auxiliary results}
\label{sec:aux}
In this section the proofs of Lemmas \ref{lemma:ESqpIS} and \ref{lemma:expApprox} are carried out. Lemma \ref{lemma:ESqpIS} follows from results for the quantile process by an application of Theorem \ref{thm:GaoZhao}. The proof of Lemma \ref{lemma:expApprox} relies on exponential bounds originating from LDP's for some one-dimensional versions of the processes considered in Section \ref{sec:MDPrisk}.
\begin{proof}[Proof of Lemma \ref{lemma:ESqpIS}]
The mapping $\gamma _{q,p} $ is linear and thus Hadamard differentiable at $T^{-1}$ with the derivative being the mapping itself evaluated at $T^-1$, $\gamma _p(T^{-1})$. By the chain rule, the mapping $T_{n} ^w \mapsto \gamma _{q,p} ((T_n ^w )^{-1})$ is also Hadamard differentiable. Hadamard differentiability together with Corollary \ref{cor:tailIS} and Theorem \ref{thm:GaoZhao} yields the LDP for $\{ b_n ( \gamma_{p,q} ((T_{n} ^w)^{-1}) - \gamma_{p,q} (T^{-1}))\}$ with speed $\lambda _n^{-2}$. 
The associated rate function is
\begin{equation}
		I^w _{q,p}(z) = \inf \Big \{ I_{\nu}(\eta) \colon \eta \in \M _b ^{\nu, 0}, \ \frac{1}{p}\int _q ^p \frac{\eta(  \cdot >T^{-1}(u) \}w)}{f(T^{-1}(u))}du = z \Big \}, \ z \in \R.
\end{equation}
To obtain an explicit expression for $I ^w _{q,p}$ the same type of convex optimization arguments as in the proof of Lemma \ref{lemma:expApprox} can be used.
First, note that for a measure $\eta$ such that $h = d\eta / d\nu \in L_2(\R,\nu)$, 
\begin{align*}
	\int _q ^p \frac{\eta(  \cdot >T^{-1}(u) \}w)}{f(T^{-1}(u))}du & = \int _q ^p \Big ( \int _{T^{-1}(u)} ^{\infty} \frac{w(x) h(x)}{f(T^{-1}(u))} \nu(dx)\Big) du \\
	&= \int _{T^{-1}(p)} ^{\infty} \Big ( \int _{T(x) \vee q} ^{p} \frac{w(x) h(x)}{f(T^{-1})(u)} du \Big ) \nu(dx).
\end{align*}
Furthermore, due to the assumption that the density $f$ is continuous and strictly positive on the interval $[q,p]$, 
\begin{align*}
	& \int _{T^{-1}(p)} ^{\infty} \Big ( \int _{T(x) \vee q} ^{p} \frac{w(x) h(x)}{f(T^{-1})(u)} du \Big ) \nu(dx) \\
	& = \int _{T^{-1}(p)} ^{\infty} w(x) h(x) \Big ( \int _{T(x) \vee q} ^{p} \bigl( \frac{d}{du}T^{-1}(u) \bigr) du \Big ) \nu(dx) \\
	& = ( T^{-1}(q) - T^{-1}(p)) \int _{T^{-1}(q)} ^{\infty} w(x) h(x) \nu(dx) \\ 
	& \quad + \int _{T^{-1}(p)} ^{T^{-1}(q)} (x - T^{-1}(p)) w(x) h(x) \nu(dx).
\end{align*}
This expression enters the optimization problem as a linear constraint. For brevity, the details of the optimization procedure are omitted and we emphasize that once the constraint has be re-written as above, there are no additional difficulties compared to the corresponding parts of the proof of Lemma \ref{lemma:expApprox}. Performing the optimization gives
\begin{align*}
	I^w _{q,p} (z) = \frac{z^2}{ 2 \sigma ^2 _{q , p}(w)}.
\end{align*}
\end{proof}
\begin{proof}[Proof of Lemma \ref{lemma:expApprox}]
For any $\delta >0$ and a fixed $0<q_m<p$, consider
\begin{align*}
	& \Prob \Big (b_n \bigl | \gamma _p ((T_{n} ^w)^{-1}) - \gamma _p (T^{-1}) - ( \gamma _{q_m , p}((T_{n} ^w)^{-1}) - \gamma _{q_m, p} (T^{-1})) \bigr | \geq \delta \Big ) \\
	&= \Prob \Big (b_n\frac{1}{p} \bigl | \int _{0} ^{q_m} \bigl( (T_{n} ^w)^{-1}(u) - T^{-1}(u) \bigr)du \bigr | \geq \delta \Big ).
\end{align*}
Following the proof of Propositon 3.4 in \cite{Hult2011}, an upper bound for this probability is given by 
\begin{align}
		& \Prob \Big (b_n\frac{1}{p} \bigl | \int _{0} ^{q_m} ((T_{n} ^w)^{-1}(u) - T^{-1}(u))du \bigr | \geq \delta \Big )	\notag \\
		& \leq 2 \Prob \Big (b_n\frac{q_m}{p} \bigl | (T_{n} ^w)^{-1}(q_m) - T^{-1}(q_m) \bigr | \geq \frac{\delta}{4} \Big ) \label{eq:es_mid} \\
		& \quad + \Prob \Big (b_n\frac{1}{p} \bigl | \int _{T^{-1}(q_m)} ^{\infty} (T_{n} ^w (x) - T(x))dx \bigr | \geq \frac{\delta}{4} \Big ) \label{eq:es_first} \\
		& \qquad + 	\Prob \Big ( b_n\frac{1}{p}T_{n} ^w (T^{-1}(q_m)) \bigl | (T_{n} ^w)^{-1}(q_m) - T^{-1}(q_m) \bigr | \geq \frac{\delta}{4} \Big ) \label{eq:es_last}.
\end{align}
We claim that each of \eqref{eq:es_mid}-\eqref{eq:es_last} is bounded from above by an exponential term that gives the correct behavior on logarithmic scale. The latter means that each of \eqref{eq:es_mid}-\eqref{eq:es_last} is bounded (individually) by an exponential term, dependent on $m$, such that when taking logarithm and multiplying with $\lambda _n ^{-2}$, the limit as $n \rightarrow \infty$ can be made as negative as desired. Such exponential bounds follow from large deviation results for the corresponding random variables. 

First consider \eqref{eq:es_mid}. A one-dimensional version of the LDP for the quantile process states that  $\{ b_n ((T_{n} ^w)^{-1}(q_m) - T^{-1}(q_m) \}$ satisfies the LDP and it follows that 
$$Ê\Prob \Big ( b_n \frac{q_m}{p} \bigl | (T_{n} ^w)^{-1}(q_m) - T^{-1}(q_m) \bigr | \geq \frac{\delta}{4} \Big ) = \textrm{exp} \Big \{- \lambda _n^2 \kappa _1 (q_m,p \delta /4 ) + o (\lambda _n^2) \Big \},$$
where
\begin{align*}
	\kappa _1 (q_m,\delta) &= \inf \Big \{ I_{\nu} (\eta) \colon \eta \in \M _b ^{\nu, 0}, \frac{q_m | \eta (I\{\cdot > T^{-1}(q_m) \}w)|}{f(T^{-1}(q_m))} \geq \delta \Big \}.
\end{align*}

Next, consider \eqref{eq:es_first}.  By Corollary \ref{cor:tailIS} and Theorem \ref{thm:GaoZhao} the sequence \\
$\{ b_n \int _{T^{-1}(q_m)} ^{\infty} (T_{n} ^w (x) - T(x))dx \}$ satisfies the LDP with speed $\lambda _n ^{-2}$ and rate function, 
\begin{equation*}
	I_2 (z) = \inf \Big \{ I_{\nu}(\eta) \colon \eta \in \M _b ^{\nu, 0} , \int _{F^{-1}(q_m)} ^{\infty} \eta (I\{ \cdot > x\} w) dx = z \Big \}, \ z \in \R.
\end{equation*}
The LDP implies that
\begin{align*}
	& \Prob \Big ( b_n \frac{1}{p} \bigl | \int _{T^{-1}(q_m)} ^{\infty} (T_{n} ^w (x) - T(x))dx \bigr | \geq \frac{\delta}{4} \Big ) \\
	& = \textrm{exp} \Big \{-\lambda _n ^2 \kappa _2 (q_m, p \delta / 4) + o (\lambda _n^2) \Big \},
\end{align*}
where
\begin{align*}
	\kappa _2 (q_m,\delta) &= \inf \Big \{ I_{\nu} (\eta) \colon \eta \in \M _b ^{\nu, 0},  \bigl | \int _{T^{-1}(q_m)} ^{\infty} \eta (I\{ \cdot > x \}w )dx \bigr | \geq \delta \Big \}. 
\end{align*}

Ê
Finally, consider \eqref{eq:es_last}. An upper bound is given by
\begin{align}
	& \Prob \Big (b_n\frac{1}{p}T_{n} ^w(T^{-1}(q_m)) \bigl | (T_{n} ^w)^{-1}(q_m) - T^{-1}(q_m) \bigr | \geq \frac{\delta}{4} \Big )  \nonumber \\
	& \leq 	\Prob \Big ( b_n \frac{1}{p} \bigl | T_{n} ^w (T^{-1}(q_m)) - q_m) \bigr | \times \bigl| (T_{n} ^w)^{-1}(q_m) - T^{-1}(q_m) \bigr | \geq \frac{\delta}{8} \Big ) \label{eq:prodProb} \\
	& \quad +	\Prob \Big ( b_n \frac{q_m}{p} \bigl | (T_{n} ^w)^{-1}(q_m) - T^{-1}(q_m) \bigr | \geq \frac{\delta}{8} \Big ) \label{eq:prodProb2}.  
\end{align}
The probability \eqref{eq:prodProb2} is \eqref{eq:es_mid} with $\delta /4$ replaced by $\delta / 8$ and can be treated precisely in the same way. To derive an upper bound for \eqref{eq:prodProb} define, for  some $\epsilon > 0$, the event $A_m = \{ |T_{n} ^w (T^{-1}(q_m)) - q_m| \geq \epsilon q_m\}$. Then,
\begin{align*}
	&\Prob \Big ( b_n \frac{1}{p} \bigl | T_{n} ^w (T^{-1}(q_m)) - q_m) \bigr | \times \bigl| (T_{n} ^w)^{-1}(q_m) - T^{-1}(q_m) \bigr | \geq \frac{\delta}{8} \Big ) \\
	& = \Prob \Big ( b_n \frac{1}{p} \bigl | T_{n} ^w (T^{-1}(q_m)) - q_m) \bigr | \times \bigl| (T_{n} ^w)^{-1}(q_m) - T^{-1}(q_m) \bigr | \geq \frac{\delta}{8} , A_m  \Big ) \\
	& \quad + \Prob \Big ( b_n \frac{1}{p} \bigl | T_{n} ^w (T^{-1}(q_m)) - q_m) \bigr | \times \bigl| (T_{n} ^w)^{-1}(q_m) - T^{-1}(q_m) \bigr | \geq \frac{\delta}{8} , \ A_m ^c  \Big ).
\end{align*}
The second term on the right-hand side is bounded from above by
\begin{align*}Ê
& \Prob \Big ( b_n \frac{1}{p} \bigl | T_{n} ^w (T^{-1}(q_m)) - q_m) \bigr | \times \bigl| T_{n,\nu}^{-1}(q_m) - T^{-1}(q_m) \bigr | \geq \frac{\delta}{8} , A_m ^c  \Big ) \\
& \leq \Prob \Big ( b_n \frac{q_m}{p} \bigl| T_{n,\nu}^{-1}(q_m) - T^{-1}(q_m) \bigr | \geq \frac{\delta}{8 \epsilon}  \Big )
\end{align*}
Since $\epsilon$ is just a constant this probability is treated in the same way as \eqref{eq:es_mid}. Moreover,
\begin{align*}
	& \Prob \Big ( b_n \frac{1}{p} \bigl | T_{n} ^w (T^{-1}(q_m)) - q_m) \bigr | \times \bigl| (T_{n} ^w)^{-1}(q_m) - T^{-1}(q_m) \bigr | \geq \frac{\delta}{8} , \ A_m  \Big ) \\
	&\leq \Prob \Big (  \frac{1}{q_m} \bigl | T_{n} ^w (T^{-1}(q_m)) - q_m) \bigr | \geq \epsilon \Big ).
\end{align*}
Since $b_n \rightarrow \infty$ and $q_m \rightarrow 0$ as $m, n \rightarrow \infty$, for each (fixed) $m$ there is an $n_m$ such that $ b_n \geq q_m ^{-1}$ for all $n \geq n_m$. Hence, taking $n$ sufficiently large,
\begin{align*}
	\Prob \Big (  \frac{1}{q_m} \bigl | T_{n} ^w (T^{-1}(q_m)) - q_m) \bigr | \geq \epsilon \Big ) \leq \Prob \Big (  b_n \bigl | T_{n} ^w (T^{-1}(q_m)) - q_m) \bigr | \geq \epsilon \Big ).
\end{align*}
The sequence $\{Êb_n \bigl( T_{n} ^w (T^{-1}(q_m)) - q_m \bigr) \}$ satisfies the LDP in $\R$ with speed $\lambda _n ^{-2}$ and rate function
\begin{align*}
	I _3 (z) = \inf \bigl \{ I_{\nu} (\eta) \colon \eta \in \M _b ^{\nu , 0} , \ \eta (I \{ \cdot > T^{-1}(q_m) \} w) = z  \bigr \}, \ z\in \R.
\end{align*}
Therefore,
\begin{align*}
	\Prob \bigl (  b_n \bigl | T_{n} ^w (T^{-1}(q_m)) - q_m) \bigr | \geq \epsilon \bigr ) = \textrm{exp} \bigl \{ -\lambda _n ^{2} \kappa _3 (q_m, \epsilon) + o(\lambda_n ^2 )\bigr \},
\end{align*}
where
\begin{align*}
	\kappa _3 (q_m, \delta) = \inf \bigl \{ I_{\nu} (\eta) \colon \eta \in \M _b ^{\nu , 0} , \ \abs {\eta(I \{ \cdot > T^{-1}(q_m) \} w)} \geq \delta   \bigr\}.
\end{align*}

It remains to show that, for $i=1,2,3$ and any $\delta >0$, $\kappa _i (q_m , \delta) \rightarrow \infty$ as $m \rightarrow \infty$. This is achieved by first solving the variational problems to get explicit expressions for the $\kappa _i$'s and then using assumptions \eqref{eq:hyp1}-\eqref{eq:hyp4}. 
For $\kappa _1$, notice that $(q_m / f(T^{-1}(q_m))) \abs{\gamma (I\{\cdot > T^{-1}(q_m) \} w)} \geq \delta$ implies that the left-hand side is either positive and $\geq \delta$ or 
negative and $\leq - \delta$; similarly for $\kappa_2$ and $\kappa _3$. Therefore, each of the variational problems involve minimization of a convex functional with convex constraints. Hence, standard arguments from convex optimization are available. For more background on convex optimization and Lagrange multipliers, see e.g.\ \cite{Luenberger}. We outline the method for $\kappa _1$, the two remaining cases are handled completely analogous and the details are therefore omitted.

First, notice that the only measures $\eta \in \M _b$ of interest are those that satisfy $\eta \ll \nu$. By the Radon-Nikodym theorem there exists for each such $\eta$ a non-negative function $h$ such that $\eta(dx) = h(x) \nu(dx)$. Hence, the optimization can be taken over functions in $L_2 (\R, \nu)$ such that $h= 0$ on $(\textrm{supp}(\nu))^c$. That the functions lie in $L_2 (\R, \nu)$ corresponds to a finite rate in the large deviation analysis. For now, let $h$ denote the Radon-Nikodym derivative of some arbitrary $\eta \in \M _b ^{\nu}$ with respect to the sampling distribution $\nu$. Start by taking the second (inequality) constraint to be $(q_m / f(T^{-1}(q_m))) \int _{ T^{-1}(q_m)} ^{\infty} w(x) h(x) \, \nu(dx) + \delta \leq 0$. Then, in the language of convex optimization, the problem of interest is
\begin{equation*}
\begin{aligned}
& \underset{h}{\text{minimize}}
& &  \frac{1}{2}\int h(x) ^2 \nu(dx) , \\
& \text{subject to}
& & \int h(x) \nu(dx) = 0 , \\
&&& \frac{q_m}{f (T^{-1}(q_m))}\int _{T^{-1}(q_m)} ^{\infty} w(x) h(x) \nu(dx) + \delta \leq 0.
\end{aligned}
\end{equation*}
Define the Lagrangian $L$ by
 \begin{align*}
 	L (h) &= \frac{1}{2}\int h(x)^2  \nu(dx) + \lambda _1 \Big ( \int h(x)\nu(dx) - 0 \Big ) \\ 
	& \quad + \lambda _2 \Big ( \int _{T^{-1}(q_m)} ^{\infty} \frac{q_m w(x) h(x)}{f (T^{-1}(q_m))} \nu(dx) + \delta \Big ),
 \end{align*} 
 for constants $\lambda_1, \lambda_2$. In order to solve the minimization problem, we note that 
 \begin{align*}
  	\lim _{\epsilon \rightarrow 0} \frac{L (h + \epsilon g) - L(h)}{\epsilon} &= \int g(x)h(x) \nu(dx) + \lambda _1 \int g(x) \nu(dx) \\
	& \quad + \lambda_2 \int _{T^{-1}(q_m)} ^{\infty} \frac{q_m g(x) w(x)}{f (T^{-1}(q_m))} \nu(dx) \\
 &= \int _{- \infty} ^{T^{-1}(q_m)} g(x) (h(x) + \lambda _1) \nu(dx) \\ 
 & \quad + \int _{T^{-1}(q_m)} ^{\infty} g(x) \Big (h(x) + \lambda_1 + \lambda_2 \frac{w(x)}{f (F^{-1}(q_m))} \Big ) \nu(dx).
 \end{align*}
For this to be equal to $0$ for all choices of $g \in L_2 (\R , \nu)$ it must hold that
\begin{equation*}
	h(x) = \begin{cases}
	-\lambda _1, & \mbox{on $ (-\infty, T^{-1}(q_m)) \cap \textrm{supp}(\nu),$} \\
	-\lambda _1 - \lambda_2 \frac{q_m w(x)}{f (T^{-1}(q_m))}, & \mbox{on $(T^{-1}(q_m), \infty) \cap \textrm{supp}(\nu).$} 
	\end{cases}	
\end{equation*}
The first constraint, fulfilled with equality, gives
\begin{align*}
	0 &= \int h(x)\nu(dx) \\
	&= - \lambda _1 \int _{- \infty} ^{T^{-1}(q_m)} \nu(dx) + \int _{T^{-1}(q_m)} ^{\infty} \Big ( - \lambda_1 - \lambda _2 \frac{w(x)}{f (T^{-1}(q_m))} \Big ) \nu(dx) \\
	& = -\lambda _1 - \lambda _2 \frac{q_m}{f (T^{-1}(q_m))}.
\end{align*}
That is,
\begin{equation}
\label{eq:lambda1}
	\lambda_1 = - \lambda _2 \frac{q_m}{f (T^{-1}(q_m))}.
\end{equation}
Similarly, the second constraint yields
\begin{align*}
	-\delta & = \int _{T^{-1}(q_m)} ^{\infty} \frac{q_m w(x) h(x)}{f (T^{-1}(q_m))} \nu(dx) \\
	& = \frac{q_m}{f (T^{-1}(q_m))} \int _{T^{-1}(q_m)} ^{\infty} \Big(- \lambda _1 - \lambda _2 \frac{q_m w(x)}{f (T^{-1}(q_m))} \Big ) w(x) \nu(dx) .
\end{align*}
After some algebra, using \eqref{eq:lambda1},
\begin{align}
\label{eq:lambda2}
	\lambda _2 &=  \frac{ \delta q_m f (T^{-1}(q_m)) ^2 }{\E _{\nu}[w(X)^2 I \{ X > T^{-1}(q_m) \}] - q_m ^2}.
\end{align}
Inserting \eqref{eq:lambda1} and \eqref{eq:lambda2} into the expression for $h$ and evaluating $L (h)$ gives the optimal value. One then proceeds in the same way when the second constraint is taken to be $\delta - (q_m / f(T^{-1}(q_m))) \int _{T^{-1}(q_m)} ^{\infty} w(x) h(x) \nu(dx) \leq 0$. It turns out that this minimization problem has the same optimal value, namely $\kappa _1$. Some algebra yields that
$$Ê\kappa _1 (q_m, \delta) = \frac{\delta ^2}{2} \frac{ f(T^{-1}(q_m))^2}{q_m ^2 (\E _{\nu} [w(X)^2 I\{ X > T^{-1}(q_m) \}] - q_m ^2)}.$$

Following the same procedure for $\kappa_2$ and $\kappa _3$,
\begin{align*}
	\kappa_2 (q_m, \delta) &= \frac{\delta ^2}{2} \Big  (T^{-1}(q_m))^2 E_{\nu}[w(X)^2I\{ X > T^{-1}(q_m) \}] \\ 
	& \quad -2 T^{-1}(q_m)E_{\nu}[X w(X)^2 I\{ X > T^{-1}(q_m) \}] \\
	& \qquad + E_ {\nu}[X^2 w(X)^2 I\{ X > T^{-1}(q_m) \}] \\
	& \quad \qquad - (E_{\mu} [X I\{ X > T^{-1}(q_m) \}] - q_mT^{-1}(q_m))^2 \Big ) ^{-1}.
\end{align*}
and
$$Ê\kappa_3 (q_m , \delta) = \frac{\delta ^2}{2} \frac{1}{\E _{\nu} [w(X) ^2 I \{ÊX > T^{-1}(q_m)\}] - q_m ^2}.$$

Finally, we must verify that assumptions \eqref{eq:hyp1}-\eqref{eq:hyp4} are indeed sufficient for $\kappa_i (q_m,\delta) \rightarrow  \infty$, as $m \rightarrow \infty$, $i=1,2,3$. Since
\begin{align*}
& \Big ( \frac{f(T^{-1}(q_m))^2}{q_m ^2 (\E _{\nu} [w(X)^2 I\{ X > T^{-1}(q_m)\}] - q_m ^2)}\Big )^{-1}  \\
&= \frac{q_m ^2 \E _{\nu} [w(X)^2 I\{ X > T^{-1}(q_m)\}]}{f(T^{-1}(q_m))^2} - \frac{q_m ^4}{f(T^{-1}(q_m))^2}, 
\end{align*}
converges to $0$ by \eqref{eq:hyp2} and \eqref{eq:hyp4}, the inverse goes to $\infty$ as $m \rightarrow \infty$. This takes care of $\kappa _1$. Moreover, every term in the denominator of $\kappa _2$ is either equal to or bounded from above by one of $\E _{\mu} [X I\{ X > T^{-1}(q_m)\}]$ and $\E _{\nu} [X^2 w(X)^2 I \{ X > T^{-1}(q_m)\}]$. The assumption \eqref{eq:hyp1} of finite (first and) second moment of $\mu$ and the assumption \eqref{eq:hyp3} on the weighted second moment under $\nu$ together imply that both terms converge to $0$ as $m\rightarrow \infty$. Hence, the denominator converges to 0 and $\kappa _2(q_m,\delta) \rightarrow \infty$ as $m \rightarrow \infty$.
Also, $\kappa _3$ trivially goes to $\infty$ as $m \rightarrow \infty$.

It follows that for any $K, \delta >0$ it is possible to pick an integer $m_K < \infty $ such that
$$Ê\limsup _{n \rightarrow \infty} \frac{1}{\lambda _n ^2} \log \Prob \Bigl ( b_n \frac{1}{p} \Big | \int _{0} ^{q_{m}} ((T_n ^w)^{-1}(u) - T^{-1}(u))du \Big | \geq \delta \Bigr ) \leq - K,$$
for all $m \geq m_K$. This completes the proof.
\end{proof}
\section*{Acknowledgment}
It is a pleasure to thank my advisor Henrik Hult for valuable discussions and his insightful comments and suggestions, as well as for his constant encouragement, throughout this work. Moreover, his feedback on preliminary drafts have greatly helped improve this manuscript.

\appendix

\bibliographystyle{plain}
\bibliography{references}

\end{document}